\documentclass[12pt]{amsart}
\usepackage[utf8]{inputenc}
\usepackage{graphicx} 

\usepackage{graphicx}
\usepackage{amssymb, amsthm, amsmath, tcolorbox}
\usepackage{enumitem}
\usepackage{tikz-cd}
\usepackage{bbold}
\usepackage{ytableau}

\usepackage{hyperref}

\theoremstyle{definition}
\newtheorem*{theorem*}{Theorem}

\newtheorem{theorem}{Theorem}[section]
\newtheorem{lemma}[theorem]{Lemma}

\newtheorem{remark}[theorem]{Remark}

\newtheorem{corollary}[theorem]{Corollary}
\newtheorem{definition}[theorem]{Definition}

\newtheorem*{conjecture}{Conjecture}

\title{Single-SEM Schubert polynomials}
\author{Dora Woodruff}
\date{}
\usepackage[margin=1in]{geometry}

\begin{document}
\maketitle

\begin{abstract}
    We prove a pattern-avoidance characterization of $w \in S_n$ such that the Schubert polynomial $\mathfrak{S}_w$ is a standard elementary monomial. This characterization tells us which quantum Schubert polynomials are easiest to compute. We solve a similar pattern-avoidance problem for complete homogeneous monomials. 
\end{abstract}

\section{Introduction}

Schubert polynomials form an important basis for the polynomial ring $\mathbb{Z}[x_1, x_2 \dots]$. For each $w \in S_n$, there is a Schubert polynomial $\mathfrak{S}_w$ in variables $x_1, \dots x_{n-1}$. It is natural to wonder when a Schubert polynomial is just one monomial. This question has a concise answer in terms of pattern avoidance: 

\begin{theorem}
    The Schubert polynomial $\mathfrak{S}_w$ is a single monomial $x_1^{l_1}x_2^{l_2}\dots x_{n-1}^{l_{n-1}}$ if and only if the following equivalent conditions hold:
    \begin{enumerate}
        \item The \emph{Lehmer code} of $w$ is nonincreasing. 
        \item $w$ avoids the pattern $132$.
    \end{enumerate}
\end{theorem}

Such permutations are called \textit{dominant}. Our first goal is to give an analogous description of when a Schubert polynomial is a single \textit{standard elementary monomial}. A standard elementary monomial (SEM) is a product
\[e^1_{a_1}e^2_{a_2}e^3_{a_3}\dots\ = e_{\vec{a}}\]
where $e^i_j$ is the degree-$j$ elementary symmetric polynomial in $i$ variables, and only finitely many of the $a_i$'s are nonzero. Fomin, Gelfand, and Postnikov \cite{quantumschubert} showed that SEMs form a basis for the polynomial ring $\mathbb{Z}[x_1, x_2 \dots]$ and that this basis has several deep connections with Schubert polynomials. 

Specifically, Fomin, Gelfand and Postnikov \cite{quantumschubert} studied SEMs in the context of \emph{quantum Schubert polynomials}. They showed that if $\mathfrak{S}_w = \sum_{\vec{a}} c_{\vec{a}} e_{\vec{a}}$ is an expansion of $\mathfrak{S}_w$ into SEMs, then the quantum Schubert polynomial $\mathfrak{S}_w^q$ can be written as $\sum_{\vec{a}}c_{\vec{a}}E_{\vec{a}}$, where $E_{\vec{a}}$ is a \emph{quantum SEM}. Essentially, if we can compute the SEM expansion of $\mathfrak{S}_w$, then we can also easily compute $\mathfrak{S}_w^q$. Asking which Schubert polynomials are single SEMs thus asks, `which quantum Schubert polynomials are easiest to compute?'

Fomin, Gelfand, and Postnikov \cite{quantumschubert} also studied the quantization of complete homogeneous monomials (CHMs). A CHM is analogously a product 

\[h^1_{a_1}h^2_{a_2}\dots = h_{\vec{a}}\]

where $h^i_j$ is the degree-$j$ homogeneous symmetric polynomial in $i$ variables and only finitely many of the $a_i$'s are nonzero. They showed that the quantization map sends a complete homogeneous polynomial to a certain determinant of quantum elementary symmetric polynomials, and more generally sends a CHM to a certain product of determinants of SEMs. Thus, CHMs are also relatively straightforward to quantize. 

A fair amount of attention has already been given to SEM expansions of Schubert polynomials. Winkel \cite{schurandschubert} gave a determinantal formula for Schur polynomials and observed some interesting patterns in SEM expansions for Schubert polynomials more generally; for instance, the coefficients tend to be small in absolute value. Hatam, Johnson, Liu, and Macaulay \cite{determinantal} gave a determinantal formula for the SEM expansion of $\mathfrak{S}_w$ when $w$ avoids a list of $13$ patterns. 

Despite this progress, SEM expansions for Schubert polynomials more generally are far from well understood: giving a cancellation-free formula for the expansion of a Schubert polynomial into SEMs remains an open problem, and many other properties of these expansions remain mysterious. 

Asking for the analogues of dominant permutations in the SEM basis is thus a natural question. 

\begin{theorem}\label{patternavoidanceSEMs}
A Schubert polynomial $\mathfrak{S}_w$ is a single standard elementary monomial if and only if $w$ avoids the patterns $312$ and $1432$. 
\end{theorem}

We also consider the analogous question in the setting of CHMs. The relationship between Schubert polynomials and CHMs has been less explored so far. 

\begin{theorem}\label{patternavoidanceCHMs}
A Schubert polynomial $\mathfrak{S}_w$ is a complete homogeneous monomial if and only if $w$ avoids the patterns $321$ and $231$. 
\end{theorem}

As a consequence, we notice that single-monomial, single-SEM, and single-CHM Schubert polynomials are all counted by nice enumerative sequences:

\begin{corollary}
    The number of $w \in S_n$ such that $\mathfrak{S}_w$ is a single monomial, standard elementary monomial, and complete homogeneous monomial, is (respectively), the Catalan number $C_n$, the Fibonacci number $F_{2n}$, and $2^{n-1}$. 
\end{corollary}

We will start with some brief background on standard elementary monomials and Schubert polynomials. Then, we prove Theorems \ref{patternavoidanceSEMs} and \ref{patternavoidanceCHMs}. Finally, we discuss further potential directions. 

\subsection{Acknowledgements}
We thank Alex Postnikov, Ilani Axelrod-Freed, Foster Tom, and Son Nguyen for interesting conversations, and especially Son Nguyen for sharing code that allowed for the computation of examples. 

\section{Background}

\subsection{Standard Elementary Monomials and Divided Difference Operators}

Schubert polynomials can be defined in terms of \emph{divided difference operators.} The divided difference operator $\partial_i$ acts by 
\[\partial_i(f) = \frac{f-s_i\cdot f}{x_i-x_{i+1}}\]
where $s_i$ is the simple transposition $(i, i+1)$ and $(s_i\cdot f)(x_1 \dots x_n) = f(x_1 \dots x_{i+1}, x_i \dots x_n)$. (So, if $f$ is symmetric in $x_i, x_{i+1}$, then $\partial_i(f) = 0$). Then, $\mathfrak{S}_w$ is defined recursively: for the longest element $w_0 \in S_n$, $\mathfrak{S}_{w_0} = x_1^{n-1}x_2^{n-2}\dots x_{n-1}$. Otherwise, we define $\partial_i(\mathfrak{S}_w) = \mathfrak{S}_{ws_i}$ if $ws_i$ covers $w$ in the weak Bruhat order (that is, $l(ws_i) = l(w)+1$), and $\partial_i(\mathfrak{S}_w) = 0$ if not. 

A \emph{descent} (respectively, ascent) of $w$ is a position $i$ such that $w_{i+1} < w_i$ (respectively, $w_{i+1} > w_i$). Notice that $\partial_i(\mathfrak{S}_w) = 0$ if and only if $i$ is an ascent of $w$. 

Both SEMs and CHMs interact nicely with divided difference operators: 

\begin{lemma}\label{divideddifferencelemma}
    We have $\partial_i(e^k_j) = 0$ if $i \neq k$, and $\partial_i(e^i_j) = e^{i-1}_{j-1}$. 

    Similarly, $\partial_i(h^k_j) = 0$ if $i \neq k$, and $\partial_i(h^i_j) = h^{i+1}_{j-1}$. 
\end{lemma}

The \emph{twisted Leibniz rule} for divided difference operators says that $\partial_i(pq) = \partial_i(p)q + s_i(p)\partial_i(q)$. From lemma \ref{divideddifferencelemma} and the twisted Leibniz rule, we can observe the following useful fact:

\begin{lemma}\label{lemma:induction}
    Suppose that $\mathfrak{S}_w$ is a single SEM, $e_{a_1 a_2 \dots a_n}$. Then, $a_i = 0$ if and only if $i$ is an ascent of $w$. Furthermore, if $i$ is a descent of $w$ and $i-1$ is an ascent, then $\partial_i(\mathfrak{S}_w) = \mathfrak{S}_{ws_i}$ is also a single SEM. 

    Similarly, suppose that $\mathfrak{S}_w$ is a single CHM, $h_{a_1 a_2 \dots a_n}$. Then, $a_i = 0$ if and only if $i$ is an ascent of $w$. Furthermore, if $i$ is a descent of $w$ and $i+1$ is an ascent, then $\partial_i(\mathfrak{S}_w) = \mathfrak{S}_{ws_i}$ is also a single CHM. 
\end{lemma}

\subsection{Permutations and Pipe Dreams}

Pipe dreams are a combinatorial model for Schubert polynomials which will be useful for us. A pipe dream is a type of \emph{wiring diagram} for a permutation $w \in S_n$. Every box of the $[n] \times [n]$ grid contains either a cross or a pair of elbows (see Figure \ref{fig:pipedreams}). 

A pipe dream corresponds to $w$ if the wire that enters at the left of row $i$ exits at the top of column $w_i$. A pipe dream is called \emph{reduced} if no two wires cross more than once. Each pipe dream $D$ is assigned a weight $\text{wt}(D)$: a cross in row $i$ is assigned the weight $x_i$, and $\text{wt}(D)$ is the product of the weights of its crosses. The following theorem of Bergeron and Billey is well-known:

\begin{theorem}[\cite{pipedreams}]
    \[\mathfrak{S}_w = \sum_{D \in RC(w)} \text{wt}(D)\]
    where $RC(w)$ is the set of all reduced pipe dreams of $w$. 
\end{theorem}

For instance, Figure \ref{fig:pipedreams} shows all reduced pipe dreams for $w = 4132$, so $\mathfrak{S}_{4132} = x_1^3x_3 + x_1^3x_2$.

\begin{figure}[h!]\label{fig:pipedreams}
\centering
\includegraphics[scale=0.35]{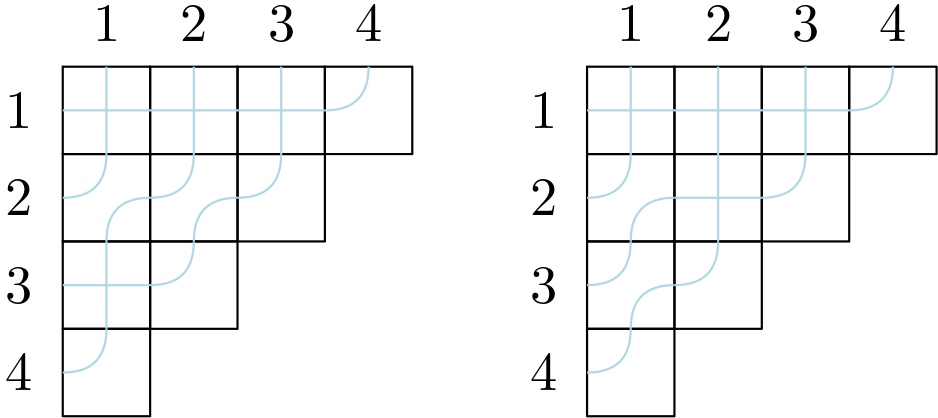}
\caption{The two (reduced) pipe dreams for $4132$. The pipe dream on the left is the \emph{bottom pipe dream} of $4132$ because its crossings are left-adjusted, and indeed, $L(4132) = (3, 0, 1, 0)$.}
\end{figure}

The \emph{Lehmer code} of $w$ will be a key definition for us:

\begin{definition}
    The \emph{Lehmer code} $L(w) = (L_1, L_2 \dots L_n)$ is given by \[L_i = |\{j > i| w(j) < w(i)\}|\] 
\end{definition}

Notice that $L_{i+1} < L_i$ if and only if $i$ is a descent of $w$ (that is, $w(i) < w(i+1)$). 

The \emph{bottom pipe dream} of $w$ is the unique left-justified pipe dream for $w$ where the number of crossings in row $i$ is $L(i)$. (Here, left-justified means that every row is a string of crossings followed by a string of elbows). For example, the bottom pipe dream for $w = 4132$ is the left pipe dream in Figure \ref{fig:pipedreams}. 

A \emph{ladder move} on reduced pipe dreams is a move as in Figure \ref{fig:laddermove}: a crossing with elbows to its right `climbs a ladder' of rows of crossings, ending with elbows to its left. We say a ladder move is of order $k$ if there are $k$ intermediate rows of crosses (so, if the moved crossing starts in row $i$, it ends in row $i-k-1$). 

For example, the ladder move in Figure \ref{fig:laddermove} is of order $4$. Ladder moves of order $0$ are called \textit{simple ladder moves}. One can see that ladder moves do not change the associated permutation or reducedness of a pipe dreams. Thus, it is natural to ask if reduced pipe dreams of a given permutation are connected by ladder moves. This was proved by Bergeron and Billey \cite{pipedreams}.

\begin{figure}[h!]
    \centering
    \includegraphics[scale = 0.3]{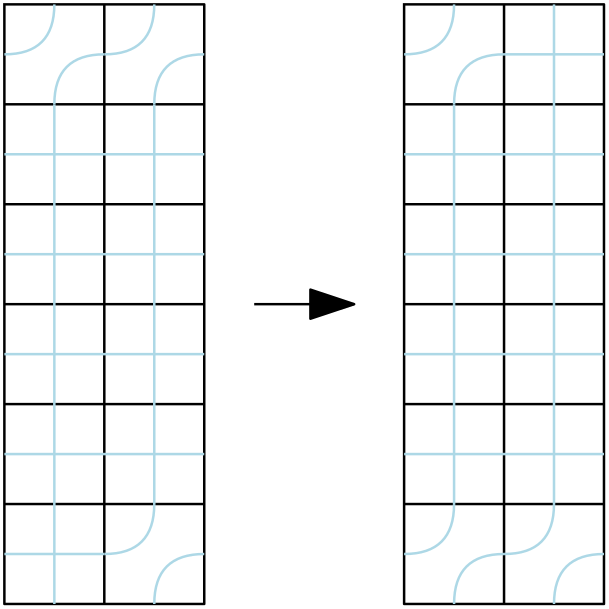}
    \caption{A ladder move of order $4$}
    \label{fig:laddermove}
\end{figure}

\begin{theorem}[{\cite[Theorem 3.7]{pipedreams}}]\label{thm:laddermove}
    Every reduced pipe dream for $w$ can be obtained from the bottom pipe dream for $w$ by a sequence of \emph{ladder moves}. 
\end{theorem}

This fact is often useful when computing Schubert polynomials using the pipe dream model. 

To prove sufficiency in Theorems \ref{patternavoidanceSEMs} and \ref{patternavoidanceCHMs}, we will make use of pipe dreams. To prove necessity, we will use \ref{lemma:induction} and \textit{Monk's rule}. Monk's rule describes how to multiply $\mathfrak{S}_w$ with $e^k_1$:

\begin{theorem}
    For any $w \in S_n$ and $k$:

    \[e^k_1\mathfrak{S}_w = \sum_{(i_1, i_2)} \mathfrak{S}_{w(i_1, i_2)}\]
    where the sum is over transpositions $(i_1, i_2)$ such that
    \begin{enumerate}
        \item $i_1 \leq k < i_2$ and 
        \item $w(i_1, i_2)$ covers $w$ in the weak Bruhat order. (That is, $l(w(i_1, i_2)) = l(w)+1$). 
    \end{enumerate}
\end{theorem}

\section{Proofs of Theorems \ref{patternavoidanceSEMs} and \ref{patternavoidanceCHMs}}

\subsection{Proof of Theorem \ref{patternavoidanceSEMs}}

To prove Theorem \ref{patternavoidanceSEMs}, we first prove the following convenient characterization of $1432$ and $312$-avoiding permutations:

\begin{lemma}\label{lehmerrules}
    A permutation $w$ avoids $1432$ and $312$ if and only if the Lehmer code $L(w) = (L_1 \dots L_n)$ satisfies the following three rules:
    \begin{enumerate}
        \item At every step, the Lehmer code decreases by at most one. That is, $L_i-L_{i+1}\leq 1$.
        \item Similarly, at every step, the Lehmer code increases by at most one. That is, $L_i-L_{i+1}\geq -1$.
        \item Between any two increases, there is at least one decrease. That is, if $L_i < L_{i+1}$ and $L_j < L_{j+1}$ for $i < j$, there is some $k$ between $i$ and $j$ such that $L_k > L_{k+1}$. 
    \end{enumerate}
\end{lemma}

Throughout, we will refer to the three rules above as `the Lehmer rules.'

\begin{proof}
    In one direction, first assume that $w$ has a $1432$ pattern. By definition, we can find indices $i < j < k < l$ with $w_i, w_j, w_k, w_l$ in relative order $1432$. If there is any $i'$ between $i, j$ with $w_{i'} < w_i$, indices $i', j, k, l$ also give us a $1432$ pattern; consider that pattern instead. 
    
    Otherwise, we have $L_j \geq L(i)+2$, because $w_j > w_k, w_l$ and $w_i < w_k, w_l$, and for any other index $i' > i$, if $w_i > w_{i'}$, then $w_j > w_{i'}$ too. But, if $w$ satisfies rules 1 and 2 as above, then it is impossible for $L_j \geq L_i+2$ for $j > i$. 

    Now, assume that $w$ has a $312$ pattern, so that we can find indices $i < j < k$ with $w_i, w_j, w_k$ in relative order $312$. We can assume that $j = i+1$: suppose some $i'$ is between $i, j$. If $w_{i'} > w_k$, $i', j, k$ is also a $312$ pattern. If $<w_{i'}<w_k$, $i, i', k$ is a $312$ pattern. So, if $j \neq i+1$, we can always choose an $i, j$ closer together.
    
    Then, for a $312$ pattern $w_i, w_{i+1}, w_k$, we have $L_i \geq L_{i+1}+2$, since $w_i > w_{i+1}, w_k$ and $w_{i+1} < w_k$. So, $w$ does not satisfy rule 3. 

    So, we have shown that if $w$ satisfies rules 1, 2, and 3 as above, then $w$ cannot contain a $1432$ or $312$ pattern. The other direction is similar; we leave the details to the reader. 
\end{proof}

\begin{remark}
    The fact that permutations in $S_n$ avoiding $1432$ and $312$ are enumerated by the Fibonacci number $F_{2n}$ was proved by West \cite{patternavoidance} using generating functions and trees. Proposition \ref{lehmerrules} gives us another simple proof of this fact. If $L(w)$ is the Lehmer code of $w$, then consider the lattice path $P$ with vertices
    \[(0,0) = (0,L_n), (1, L_{n-1}) \dots (n-1, L_1)\]

    By Lehmer rules 1 and 2, $P$ is a \emph{Motzkin meander}: a lattice path starting at $(0, 0)$ taking steps $D = (-1, 1), U = (1, 1), H = (1, 0)$ that never crosses under the $x$-axis. $P$ satisfies the additional condition, coming from Lehmer rule 3, that between between any two $D$ steps there is at least one $U$ step. Any such path corresponds to the Lehmer code of a unique permutation $w$, since such a $P$ never increases above the line $y=x$. 

    To choose such a $P$, choose $k$ steps in the lattice path to be $H$-steps. Of the remaining steps, assign the first one to be $U$ (otherwise, the path goes below the $x$-axis). There are $n-k-1$ steps left, and we can choose any non-adjacent subset of the remaining steps to be $D$. The number of such subsets is $F_{n-k-1}$. Thus, the number of such $P$ is 
    \[\sum_{k=0}^n {n \choose k} F_{n-k-1} = F_{2n}\]
\end{remark}

Theorem \ref{patternavoidanceSEMs} will follow from the following three lemmas: Lemma \ref{lemma:semsufficient} shows that for $\mathfrak{S}_w$ to be an SEM, it is sufficient for $w$ to satisfy Lehmer rules $1, 2$, and $3$. Lemma \ref{lemma:semnecessary1} shows that Lehmer rule 1 is necessary. Finally, Lemma \ref{lemma:semnecessary2} shows that Lehmer rules 2 and 3 are necessary. 

\begin{lemma}\label{lemma:semsufficient}
    If $w$ satisfies the Lehmer rules, then $\mathfrak{S}_w$ is a standard elementary monomial. 
\end{lemma}

\begin{proof}
    If $w$ is a dominant permutation, then the statement is clear: recall that if $w$ is dominant, then $L(w)$ is nonincreasing and $\mathfrak{S}_w = x^{L(w)}$.  
    Then, we have 
    \[\mathfrak{S}_w = x^{L(w)} = \prod_{i \in D(w)} e_i^i\]
    where $D(w)$ is the descent set of $w$. 

    Otherwise, consider the bottom pipe dream $P$ of $w$.  There is a unique \emph{dominant} permutation $w'$ with bottom pipe dream $P'$ satisfying the conditions:
    \begin{enumerate}
        \item The Lehmer code of $w'$ decreases by at most $1$ at each step and
        \item We can obtain $P$ from $P'$ by adding at most one crossing to the end of each row of crossings in $P'$. No crossing is added to the end of a row $i$ where $P$ has more crossings in row $i$ than in row $i+1$. 
    \end{enumerate}

(That is, $P$ is almost the bottom pipe dream of a dominant permutation satisfying the Lehmer rules with a few extra crossings). Call the crossings we add to $P'$ \emph{outer crossings} of $P$. The outer crossings are grouped together in width one vertical columns, which we call \emph{outer columns}. See figure \ref{fig:pipedreamSEM} for an example. 

Theorem \ref{thm:laddermove} tells us that every reduced pipe dream for $w$ can be obtained from $P$ by applying a sequence of ladder moves. Here, the only ladder moves that can ever be performed are \emph{simple} ladder moves applied to the outer crossings. That is, if we imagine diagonal rails extending out from each outer crossing, the only pipe dreams for $w$ are obtained by sliding each outer crossing along the rails. 

Furthermore, suppose that in $P$, crossings in squares $(i, j), (i', j)$ with $i' < i$ lie in the same outer column. Then, the first crossing will always be in a row strictly higher than the row of the second crossing after any sequence of simple ladder moves (the second crossing can never pass next to the first under any valid ladder move). 

However, outer crossings that start in different outer columns do not interact with each other. Specifically, after any sequence of ladder moves, two crossings that started in different outer columns are never in adjacent squares.  

From these observations, we see that each outer column contributes a factor of $e^i_j$, where $i$ is the position of the lowest crossing in the outer column, and $j$ is the number of crossings in the outer column. Since different outer columns do not interact with each other, simply multiplying these factors, along with the SEM associated to the dominant $w'$, gives us the SEM $\mathfrak{S}_w$. Figure \ref{fig:rails} gives an example. 

\end{proof}

\begin{figure}\label{fig:pipedreamSEM}          \begin{ytableau}
            $+$ & $+$ & & &\\
            $+$ & $+$ & $\textcolor{red}{+}$ & &\\
            $+$ & $+$ & & &\\
            $+$ & & & & \\
            $+$ & $\textcolor{red}{+}$ & & & \\
            $+$ &$\textcolor{red}{+}$ & & & \\
            $+$ & & & & \\
            & & & & \\
        \end{ytableau}
    \caption{Above is the bottom pipe dream $P$ for $35427861$. (Cross signs denote squares with crossings; empty squares denote squares with elbows). We can construct $P$ by starting with the bottom pipe dream $P'$ of $34526781$ (in black) and adding the outer crossings (in red) to the outermost edge of $P'$. There are two outer columns of length 1 and 2.}

\end{figure}

\begin{figure}\label{fig:rails}
\includegraphics[scale=0.85]{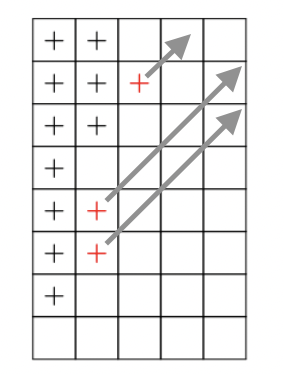} 

\caption{The grey arrows represent the possible places where we can slide each red outer crossing. In this example, the outer columns contribute factors of $e^2_1$ and $e^6_2$ each. Then, the dominant bottom pipe dream $P'$ contributes $e^3_3e^7_7$, so we get $\mathfrak{S}_{35427861} = e^2_1e^3_3e^6_2e_7^7$.}
\end{figure}

\begin{lemma}\label{lemma:semnecessary1}
    If $w$ breaks rule 1 in the statement of Lemma \ref{lehmerrules}, then $\mathfrak{S}_w$ cannot be a standard elementary monomial.  
\end{lemma}

\begin{proof}
Suppose that $\mathfrak{S}_w = e_{\vec{a}}$; then, consider the maximal monomials (in the reverse lexicographic ordering) of both sides. The maximal monomial of $\mathfrak{S}_w$ is just $x^{L(w)}$. Meanwhile, the maximal monomial of $e_{\vec{a}}$ is obtained by multiplying the maximal monomials of each $e^i_{a_i}$. For each $j$, there is at most one $i$ where the maximal monomial of $e^i_{a_i}$ has a factor of $x_j$ and not $x_{j+1}$, namely, $i=j$. Therefore, $L_j \leq L_{j+1}+1$, as desired. 
\end{proof}

\begin{lemma}\label{lemma:semnecessary2}
    If $w$ breaks rules $2$ or $3$ in the statement of Lemma \ref{lehmerrules}, then $\mathfrak{S}_w$ cannot be a standard elementary monomial. 
\end{lemma}
\begin{proof}
    We will induct on the length of $w$. For contradiction, suppose that $L_{i+1}-L_i = k > 1$, but $\mathfrak{S}_w$ is a single SEM $e_{\vec{a}}$, in violation of Lehmer rule 2. By Lemma \ref{lemma:semnecessary1}, there must be at least $k$ positions after position $i$ where $L(w)$ decreases by $1$, and each such position is a descent of $w$. 

    First, suppose that $i+1, i+2, \dots i+k$ are not all descents of $w$. Then, we can find some $j > i+1$ such that $j$ is a descent of $w$ and $j-1$ is not. By Lemma \ref{lemma:induction}, $\partial_j(\mathfrak{S}_w) = \mathfrak{S}_{ws_j}$ is still a single SEM. However, the Lehmer code of $ws_j$ still violates rule 2, and induction on the length gives a contradiction. 

    Now, we may assume that $i+1, i+2 \dots i+k$ are all descents of $w$, so that $L(w)$ has the form 
    \[(\dots, L_i, L_i+k, L_i+k-1, L_i+k-2 \dots L_i+1, L_i \dots)\]

    Let $w' = ws_{i+1}$. Since $i+1$ is a descent of $w$ and $i$ is not, $\mathfrak{S}_{w'}$ is a single SEM by Lemma \ref{lemma:induction}. Moreover, if $\mathfrak{S}_w = e_{\vec{a}}$, we have $a_{i+1} = 1$. We must therefore have 
    \begin{equation}\label{equation1}
        e_1^{i+1}\mathfrak{S}_{w'} = \mathfrak{S}_w
    \end{equation}

    On the other hand, Monk's rule tells us that 
    \[e_1^{i+1} \mathfrak{S}_{w'} = \sum_{(i_1, i_2)} \mathfrak{S}_{w'(i_1, i_2)}\]
    where the sum is over $i_1 \leq i+1 < i_2$ such that $w'(i_1, i_2)$ covers $w'$ in the weak Bruhat order. 

    We claim that there are at least two summands on the left hand side of $e_1^{i+1}\mathfrak{S}$, contradiction equation \ref{equation1}. Indeed, we could take either $i_1 = i+1$ and $i_2 = i+2$, or $i_1 = i$ and $i_2 = i+k+1$. (To see why $w'(i, i+k+1)$ covers $w'$ in the weak Bruhat order, notice that $w_i < w_{i+k+1}$ since $l_i = l_{i+k+1}$. Furthermore, for all $i < j < i+k+1$, we have $w_j > w_{i+k+1}$ (since $j$ is a descent of $w$ for all such $j$), so that no $w_j$ is between $w_i$ and $w_{i+k+1}$). 

    So, we have shown that if $w$ violates Lehmer rule 2, then $\mathfrak{S}_w$ is not a single SEM. For rule 3, the proof is very similar. Suppose that $L_i < L_{i+1}$, $L_j < L_{j+1}$, and for all $i < k < j$, $L_k = L_{k+1}$. Similarly to before, we may assume that $j+1$ and $j+2$ are both descents of $w$, and we consider $w' = ws_{j+1}$. Again, $\mathfrak{S}_{w'}$ must be a single SEM, so that we must have 
    \begin{equation}
        e_1^{j+1}\mathfrak{S}_{w'} = \mathfrak{S}_w
    \end{equation}
    Here, when we apply Monk's rule, we could take either $(i_1, i_2) = (j+1, j+2)$ or $(i, j+3)$. By the same logic, $w$ cannot violate rule 3 if $\mathfrak{S}_w$ is a single SEM. 
\end{proof}

We guess that a stronger fact is true:

\begin{conjecture}
The SEM expansion of $\mathfrak{S}_w$ has all nonnegative coefficients only when $\mathfrak{S}_w$ is a single SEM.
\end{conjecture}

It is automatic that if $\mathfrak{S}_w$ is a nonnegative sum of SEMs, then $w$ is $312$-avoiding by the same argument as Lemma \ref{lemma:semnecessary1}. So, it only remains to show that $w$ must be $1432$-avoiding. We remark that this conjecture is the exact opposite of what happens for usual monomials, since Schubert polynomials are notably \emph{always} monomial-positive. 

\subsection{Proof of Theorem \ref{patternavoidanceCHMs}}

Now, we prove Theorem \ref{patternavoidanceCHMs}. 

\begin{lemma}\label{321/231claim}
Suppose that $w$ avoids patterns $321$ and $231$. Then, the bottom pipe dream $P$ of $w$ satisfies the following property: choose the leftmost crossing of any row of $P$ and draw a diagonal line upwards and rightwards from that crossing. This line never intersects a square containing a crossing in $P$, nor does any square immediately to the left of this line contain a crossing in $P$. 
\end{lemma}

See Figure \ref{fig:chmrails} for an example. 

\begin{figure}  \label{fig:chmrails} \includegraphics[scale=0.8]{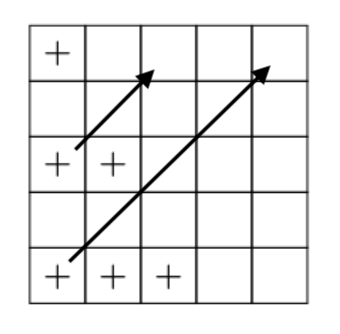}
\caption{Above is the bottom pipe dream for $1427356$ (cross signs denote crossings, and empty boxes denote non-crossings). If we draw a diagonal line out from the first crossing in row 3, then this diagonal line never intersects or is directly to the right of another crossing. However, the diagonal line emitting from the first box in row 5 enters the square directly to the right of the last crossing of row 3. Indeed, the subsequence $473$ is a $231$ pattern. 
}
\end{figure}

\begin{proof} For contradiction, let $i$ the largest row index such that the line emanating from the leftmost crossing of row $i$ intersects or is directly to the right of a crossing in row $j$. Since $L(i) > 0$ and $i$ was the largest such row, we must have $L(i+1) = 0$, and thus, since $L(i) > L(i+1)$, we have $w(i) > w(i+1)$. But then, we claim that $w(j) > w(i+1)$. Either $w(j) > w(i)$, in which this follows by transitivity, or $w(j) < w(i+1)$, in which case, $L(i+1) \geq L(j)+(i-j)+1 > 0$. Either way, we have either a $231$ or a $321$ pattern in $w$ given by the indices $j, i, i+1$. So, we have proved the claim. 
\end{proof}

Lemma \ref{321/231claim} allows us to prove that $321$ and $231$-avoidance are sufficient:

\begin{lemma}
    If $w$ avoids $231$ and $321$, then $\mathfrak{S}_w$ is a single CHM. In particular, $\mathfrak{S}_w = h_{L(w)}$. 
\end{lemma}
\begin{proof}
    We use again the fact that all pipe dreams are obtained from the bottom pipe dream by ladder moves. Here, again all of the ladder moves we can perform are simple ladder moves that just move crossings along their diagonals. Two crossings in the same row of the bottom pipe dream can never slide past each other, and any two rows can slide independently by Lemma \ref{321/231claim}. Thus, row $i$ contributes a factor $h^i_{L(i)}$, and multiplying these factors gives us $\mathfrak{S}_w = \prod_{i}h^i_{L(i)}$.  
\end{proof}

Notice that, unlike the case of SEMs and analogously to the case for usual monomials, the maximal monomial in the CHM expansion of $\mathfrak{S}_w$ is always $h_{L(w)}$.

Finally, we prove that $321$ and $231$ avoidance are necessary conditions in order for $\mathfrak{S}_w$ to be a CHM.

\begin{lemma}
    If $w$ contains a $321$ pattern, then $\mathfrak{S}_w$ is not a CHM. 
\end{lemma}

\begin{proof}
    We induct on the length of $w$. First, suppose that $w$ contains a $321$ pattern and $\mathfrak{S}_w$ is a single CHM. Let $i < j < k$ be indices such that $w_i > w_j > w_k$. First, we show how to reduce to the case where $i, j, k = i, i+1, i+2$. 

    If $i$ is an ascent of $w$, choose new indices $i+1, j, k$ which also give us a $321$ pattern. If $i$ and $i+1$ are both descents of $w$, then we have a consecutive $321$ pattern $i, i+1, i+2$. Otherwise, suppose that $i$ is a descent of $w$ and $i+1$ is an ascent. We may apply Lemma \ref{lemma:induction} to conclude that $\partial_i(\mathfrak{S}_w) = \mathfrak{S}_{wt_i}$ is also a single CHM. However, $wt_i$ contains either a $321$ pattern (if $i+1<j$, then indices $i+1, j, k$ give us a $321$ pattern in $wt_i$) or a $231$ pattern (if $i+1=j$, then $i, i+1, k$ are the indices of a $231$ pattern in $wt_i$). Inducting on the length of $w$, we get a contradiction.

    Therefore, we may assume that we have a $321$ pattern given by indices $i, i+1, i+2$. We may also assume that $i+2$ is not a descent of $w$ (otherwise, indices $i+1, i+2, i+3$ also give a $321$ pattern. Consider that pattern instead). By Lemma \ref{lemma:induction}, $\partial_{i+1}(\mathfrak{S}_w) = \mathfrak{S}_{wt_{i+1}}$ is also a single CHM. 

        If $w_{i+3} < w_{i+1}$, then we can still find a $321$ pattern in $wt_j$, given by indices $i, i+2, i+3$. By induction on the length of $w$, this is a contradiction. Otherwise, $i+2$ is an ascent of $wt_j$. In this case, using Lemma \ref{lemma:induction} we deduce that we must have
    \[h^{i+1}_1\mathfrak{S}_{wt_j} = \mathfrak{S}_w\]
    However, similarly to before, this is impossible by Monk's Rule. Namely, in the expansion of $h_1^{i+1}\mathfrak{S}_{wt_j}$ into Schubert polynomials, there a summand corresponding to the transposition $(i+1, i+2)$; there is also at least one more summand given by $(i, i')$, where $i'$ is the first index greater than $i+2$ such that $w_{i'} > w_i$.
\end{proof}

\begin{lemma}
    If $w$ contains a $231$ pattern, then $\mathfrak{S}_w$ is not a CHM. 
\end{lemma}

\begin{proof}
Let $i < j < k$ be indices such that $w_k < w_i < w_j$. By similar reasoning to before, we may assume that $j=i+1$ and $k=i+2$ by considering whether $i$ is an ascenet of $w$, $i, i+1$ are both descents of $w$, or $i$ is a descent of $w$ and $i+1$ is an ascent; we leave these details to the reader. 

    Now, given that $i, i+1, i+2$ form a $231$ pattern, we may also assume that $i+2$ is not a descent of $w$; otherwise, $i+1, i+2, i+3$ form a $321$ pattern, which we have already ruled out. By lemma \ref{lemma:induction}, we deduce that we must have
    \[h^{i+1} \cdot \mathfrak{S}_{wt_{i+1}} = \mathfrak{S}_w\]
    But once again, we get a contradiction to Monk's Rule. Namely, $w$ covers $wt_{i+1}$ in the weak Bruhat order, but so does $wt_{i+1}t_{(i, i+2)}$. Thus, the expansion of $h^{i+1} \cdot \mathfrak{S}_{wt_{i+1}}$ in the Schubert basis has at least two summands and cannot be equal to $\mathfrak{S}_w$. By contradiction, $w$ cannot contain a $231$ pattern. 
\end{proof}

\section{Further Directions}

Why does the pattern $1432$ appear in Theorem \ref{patternavoidanceSEMs}? Recall that in our proof of Lemma \ref{lemma:semsufficient}, the only ladder moves we could apply to reduced pipe dreams of $w$ were \emph{simple} ladder moves. Gao \cite{principalspecializations} showed more generally that $w$ avoids the pattern $1432$ if and only if all reduced pipe dreams for $w$ are related by simple ladder moves. Motivated by this fact, we ask the following questions:

\begin{enumerate}
    \item Can we find a cancellation-free formula for the SEM expansion of $\mathfrak{S}_w$ in the case where $w$ is $1432$-avoiding, perhaps in terms of pipe dreams?
    \item Is there a bound in terms of the number of $1432$ patterns in $w$ on the number of terms in the SEM expansion of $w$? (Or in terms of the number of Lehmer rule violations, or in terms of the number of $312$ patterns?)
    \item In particular, is there a nice description of $w \in S_n$ such that the SEM expansion of $\mathfrak{S}_w$ has only two terms?
\end{enumerate}

\bibliographystyle{plain}
\bibliography{ref}

\end{document}